\documentclass[12]{amsproc}
\usepackage[colorlinks=true,linkcolor=blue,citecolor=red]{hyperref} %

\usepackage{amsfonts,amsmath,amssymb}

\usepackage{graphicx}
\usepackage[utf8]{inputenc}
\usepackage{mathrsfs}

\usepackage{tikz}
\usetikzlibrary{ intersections,calc,positioning,arrows, shapes,quotes,math}

\newcommand{\R}{{\Bbb R}}

\def\half#1#2{\begin{matrix}\frac{#1}{#2}\end{matrix}}
\DeclareMathOperator{\Equ}{E}
\DeclareMathOperator{\Fix}{Fix}

\newtheorem{theorem}{\rm\bf Theorem}[section]
\newtheorem{proposition}[theorem]{\rm\bf Proposition}
\newtheorem{lemma}[theorem]{\rm\bf Lemma}
\newtheorem{corollary}[theorem]{\rm\bf Corollary}

\newtheorem*{theorem*}{Theorem}
\newtheorem*{theorem 1}{\rm\bf Proposition 1}
\newtheorem*{theorem 2}{\rm\bf Proposition 2}

\theoremstyle{definition}

\theoremstyle{remark}

\newtheorem{example}[theorem]{\rm\bf Example}

\begin{document}

\title{Global stability and persistence of complex foodwebs}




\author[1]{V. Kozlov}
\address{Department of Mathematics, Link\"oping University} \email{vladimir.kozlov@liu.se}           
\author{V. Tkachev}
\address{Department of Mathematics, Link\"oping University}
\email{vladimir.tkatjev@liu.se}           
\author{S. Vakulenko}
\address{St. Petersburg National Research University of Information Technologies, Russia}
\email{vakulenfr@gmail.com}
\author{U. Wennergren}
\address{Department of Physics, Chemistry, and Biology, Link\"oping University}
\email{uno.wennergren@liu.se}

\date{Received: date / Accepted: date}

\maketitle

\begin{abstract}
  We develop a novel approach to study the global behaviour of large foodwebs for ecosystems  where several species share multiple resources. The  model extends and generalize some previous works and takes into account self-limitation. Under certain conditions, we establish the global convergence and persistence of solutions.


\end{abstract}

\maketitle

\section{Introduction} \label{intro}

To mathematically show the existence and stability of large foodwebs, large and complex as foodwebs in nature, is still one of the key problems in theoretical ecology. A specific part of this theoretical issue is that many species can share just a few resources (for example ocean ecosystems including thousands of phytoplankton species) yet the competitive exclusion principle \cite{Hardin,Volterra}  asserts that such foodwebs should not exist. To partly explain that paradox \cite{Huis1999} showed that a system consisting of a single resource and three species can support chaotic dynamics where all species coexist. Another explanation of the paradox was proposed in \cite{Roy} where self-limitation effects has been taken into account.

In this paper we add complexity to the work of \cite{Huis1999} and \cite{Roy} by extending the dynamical equations considered in \cite{Roy} with self-limitation effects \cite{Roy,Alles,Alles1,Alles2,KVU16,KVU17a})  (see also a turbidostat model in \cite{Smith2003}). We obtain a complete description of the large time behaviour of the system. In particular, we explore the range in the parameter space that leads the system to the global stable equilibrium point. Furthermore,  we show that if the self-limitations exceed some critical values then the system exhibits either global stability or persistence, see Corollaries~\ref{cor1} and ~\ref{cor2}.

Traditionally, the Lyapunov function approach is used to establish global stability, see a recent review in \cite{Hsu05}. In our case, however, an explicit information about equilibrium points  is not available. Instead, we transform our problem to a finite dimensional nonlinear fixed point problem for an appropriate non-increasing operator. We show that the asymptotic behaviour of a generic solution to the initial  problem is well controlled by iterations of the introduced operator. This allows us  to derive explicit a priori estimates (see Theorem~\ref{th1} below) and the global stability.

The paper is organized as follows. In Section~\ref{model}, we present the model with self-limitations and in Section~\ref{sec:prelim} we obtain some preliminary results. We review some elementary facts on period-two-points of non-increasing maps in Section~\ref{sec:period} and discuss the structure and stratification of equilibrium points in Section~\ref{sec:equil}. In particular, in Section~\ref{sec:fixed} we consider the so-called special equilibrium points which are significant for the large time behaviour of the original dynamical system. Here we also define the corresponding finite dimensional fixed point problem. To study its dynamics and convergence we need to suitably polarize the fixed point problem. This allows us to establish bilateral estimates for the corresponding $\omega$-limit set. The main result of this section is contained in Corollary~\ref{cor1a},  which gives a sufficient condition for the existence of a unique fixed point.   In Section~\ref{sec:main} we return to the main dynamical system formulate and prove the main results on the large behaviour of the original dynamic system. In particular, we obtain some explicit conditions when the system obeys  strong persistence. Finally, in Section~\ref{sec:dis} we briefly discuss our results and relate them to some previous research.

\section{The model}  \label{model}

Given $x,y\in \mathbb{R}^n$ we  use the standard vector order relation: $x\le y$ if $x_i\le y_i$ for all $1\le i\le n$,
$x< y$  if $x\le y$ and $x\ne y$, and $x\ll y$ if $x_i< y_i$ for all $i$; $\R_{+}^n$ denotes the nonnegative cone $ \{x\in\R^n:x\ge 0\}$ and for $a\le b$, $a,b\in\R^n$
$$
[a,b]=\{x\in \R^{n}:a\le x\le b\}
$$
is the closed box with vertices at $a$ and $b$.

We consider the model where the population dynamics of $M$ species competing for $m$ complementary resources is governed by chemostat-like equations
\begin{align}
     \frac{dx_j}{dt}&=x_j ( \phi_j(v)- \mu_j -  \gamma_{j} \; x_j), \quad j=1,\dots, M
    \label{HX1}\\
          \frac{dv_i}{dt}&=D_i(S_i -v_i)   -  \sum_{j=1}^M c_{ ij} \; x_j \; \phi_j(v), \quad i=1,\dots, m
    \label{HV1}\\
 \label{Idata}
x(0)&\gg 0, \qquad
0\le v(0)\le S:=(S_1, ..., S_m),
\end{align}
Here $x_j(t)$ is species $j$ abundance  and $v_i(t)$ is the concentration of resource $i$ at time $t$, $\mu_j$ are the species mortalities, $S_i$ is the supply of resource $i$, $c_{ij} >0$ is the content  of resource $i$ in species $j$ (growth yield constants), $D_i$ is the rate of exchange of resource $i$, resource turnover (or dilution) rate),  $\gamma_{j} >0$ is a self-limitation constant of species $j$. We shall assume that the specific growth rates  $\phi_j$ are bounded Lipschitz functions subject to  the following standard conditions:
\begin{align}
           \phi_j(v)& =0 \,\,\Leftrightarrow\,\, v \in \partial \R^m_{+};
\label{MM2b}\\
    \phi_j(v)& \text{ is a nondecreasing  function of each $v_i$}.
\label{MM5}
     \end{align}
 The most relevant for biological applications example of the specific growth functions $\phi_j$ is given by the Monod equation and Liebig's ‘law of the minimum’
\begin{equation}
\label{Liebm}
\phi_j(v) =\min\left(\frac{r_jv_1}{K_{1j} +  v_1}, \ldots,  \frac{r_jv_m}{K_{mj} +  v_m} \right),
\end{equation}
where $r_j$ is the maximum specific growth rate of species $j$ and $K_{ij}$ is the half-saturation constant for resource $i$ of species $j$. Obviously, the functions \eqref{Liebm} meet the above conditions.

In absence of self-limitation ($\gamma_j=0$), the present model  naturally appears in the bioengineering context \cite{Armstrong80} and was  extensively studied for $m\le 2$ resources for both equal resource turnover rates $\mu_j=D_i=D$ in \cite{HsuHW77}, \cite{Huis1999}, \cite{HuWe2001} and for different the removal rates $\mu_j$ in \cite{Hsu78}, \cite{Hsu05}, \cite{LiSmith}, \cite{Tilman1980}, see also a recent review in \cite{Smith2003review}. For a single resource $m=1$, the dynamics of the standard model in absence of self-limitation is completely determined by the break-even concentrations $R_j$ defined as $\phi_j(R_j)=\mu_j$, see \cite{Armstrong80}, \cite{Hsu05}. For example, if the \textit{lowest break-even concentration}
\begin{equation}\label{rstarr}
R^*=\min\{\,R_1\,\ldots, R_m\}
\end{equation}
achieves on a single species $k$ then 
$$
\lim\limits_{t\to \infty} v_1(t)=R^*, \quad \lim\limits_{t\to \infty} x_{k}(t)=\frac{1}{\mu_1}(S_1-R^*)
$$ 
while $\lim_{t\to \infty} x_j(t)=0$  for all $j\ne k$. However, if $m\ge 3$, the behaviour becomes much more involved. Recent numerical simulations \cite{Huis1999}, \cite{HuWe2001} strongly support the possible chaos scenario for $(m,M)=(3,6)$ or $(5,6)$. An important step was done by Li \cite{Li2001} who established the existence of the limit cycle for $m=M=3$.

\section{Preliminaries}\label{sec:prelim}
In what follows, we shall assume that $\gamma_j>0$.

\begin{proposition}
\label{pro:1}
Solution $(x(t), v(t))$ of \eqref{HX1}, \eqref{HV1}, \eqref{Idata}  is well-defined and bounded for all  $t\ge 0$ and
\begin{align}
 &0\le x_j(t) \le x_j(0)  \left(e^{(\mu_j-\phi_j(S)) t} +\frac{1-e^{(\mu_j-\phi_j(S)) t}}{\phi_j(S) - \mu_j}\gamma_j x_j(0)\right)^{-1},
\label{estx}
\\
& 0\le v_i(t) \le    S_i (1- e^{-D_it}) +   v_i(0)  e^{-D_it},
\label{estv}
     \end{align}
If $\phi_j(S)=\mu_j$,  \eqref{estx} should be replaced by $0\le x_j(t) \le\frac{x_j(0)}{1+ \gamma_j x_j(0)t}$. In particular, $[0,S]$
is an invariant subset. Furthermore, if $\phi_j(S)\le \mu_j$ for some $i$ then $\lim_{t\to \infty}x_j(t)=0.$
\end{proposition}

\begin{proof}
Note that by \eqref{HX1}, $x_j(t)$ never vanishes unless $x_j(t)\equiv 0$. In particular, by \eqref{HX1}  $x(t)\gg 0$ as long as $x(t)$ is defined.  Furthermore, if $h_i(x,v)$ denote the right hand side of \eqref{HV1} then by \eqref{MM2b} $h_i(x,v)=D_iS_i>0$ for any $v\in \partial \R^m_{+}$, thus $v_i(0)\ge 0$ implies that $v_i(t)>0$ for all admissible $t>0$, see Proposition ~2.1 in \cite{Haddad}. Similarly, $h_i(x,S)< 0$ (unless $x=0$) and $v(0)\le S$ yields $v(t)\le S$, and thus  \eqref{Idata} and \eqref{MM5} imply $0\le \phi_j(v) \le\phi_j(S)$. This proves that $\R^{M}_+\times [0,S]$ is an invariant subset for \eqref{HX1}, \eqref{HV1}, \eqref{Idata} . Furthermore, $x_j(t) \le y_i(t)$, where
$y_i(t)$ is the solution of the Cauchy problem
$$
\frac{dy_i}{dt}=y_i(  \phi_j(S) -\mu_j- \gamma_j y_i),     \quad y_i(0)=x_j(0), \quad 1\le i\le M.
$$
This readily yields \eqref{estx} and    $\sum_{i=1}^M c_{ki} x_j\phi_j(v)\ge0$ yields the upper estimate in \eqref{estv}. Since $(x(t),v(t))$ is a bounded solution, it is well-defined for all $t\ge0$. Finally, if $\phi_j(S)\le \mu_j$ then \eqref{estx} implies $\lim_{t\to\infty}x_j(t)=0$.
\end{proof}

Proposition~\ref{pro:1}  shows that the extinction dynamics of \eqref{HX1}, \eqref{HV1}, \eqref{Idata}  depends on the sign of $\phi_j(S)-\mu_j$: for species $i$ to survive, its   specific growth rate $\phi_j(S)$ at the supply point $S$ must exceed its specific mortality rate $\mu_j$. To eliminate the trivial extinctions, we shall assume in what follows  that the \textit{survivability condition} holds:
\begin{equation}\label{survav}
\quad \phi_j(S)>\mu_j\quad \text{ for all }j.
\end{equation}
For the Monod-Liebig model \eqref{Liebm}, the survivability condition \eqref{survav} is equivalent to $0\ll R^{(j)}\ll S,$ where $R^{(j)}:=(R_{1j},\ldots, R_{mj})$ and $R_{ij}:=\frac{\mu_j}{r_j-\mu_j}K_{ij}$ is the resource requirement of a species $j$ for a resource $i$ \cite{HuWe2001}.

Below we summarize some elementary observations which will be  used throughout the paper.

\begin{lemma}
\label{lem:elem}
Let $f(x)\not\equiv0$ be continuous nonnegative and nondecreasing on $[0,S]$, $f(0)=0$. Then $S-x=f(x)$ has a unique solution $0<x_f<S$. If $f(x)\ge g(x)$ ($f(x)> g(x)$ resp.) are nonnegative  functions then $x_f\le x_g$ ($x_f< x_g$ resp).
\end{lemma}

\begin{proof}
An idea of the proof is clear from the figure below.

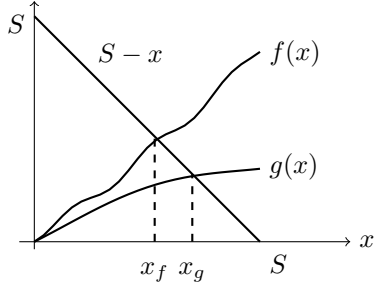
\begin{center}
\begin{tikzpicture}[scale=1]
\node[right] (B) at (3,2.5) { $f(x)$};
\node[right] (A) at (3,1) { $g(x)$};
\node[right] (A) at (0.72,2.5) { $S-x$};

\draw[->,name path=xaxis] (-0.2,0) -- (4.2,0) node[right] {$x$};
\draw[->,name path=yaxis] (0,-0.1) -- (0,3.2) node[above] {};

\draw[ thick,name path=plot,domain=0:3] plot (\x,{0.5*sin(deg(\x))+0.1*(\x)^2});

\draw[ thick,name path=plot,domain=0:3] plot (\x,{0.5*sin(deg(\x))+0.15*(sin(3*deg(\x)))^2+0.27*(\x)^2});
\draw[ thick,color=black,dashed,name path=plot,domain=-0.0:1.45] plot (1.6,\x);
\node[] at (1.6,-0.4) { $x_f$};
\draw[ thick,name path=plot,domain=0:3] plot (\x,{3-\x});
\draw[ thick,color=black,dashed,name path=plot,domain=-0.0:0.9] plot (2.1,\x);
\node[] at (2.1,-0.4) { $x_g$};
\node[right] at (3,-0.3) { $S$};
\node[left] at (0,2.9) { $S$};
\end{tikzpicture}
\end{center}
\end{proof}

\begin{lemma}\label{lem:elem2}
Let $v'(t)=F(v(t),t)$ and $\tilde{v}'(t)=\widetilde{F}(\tilde{v}(t),t)$, $t\in [0,T]$, where $F(z,t)$ and $\tilde{F}(z,t)$ are decreasing functions of $z$ for each $t$, $F(z,t)\ge \widetilde{F}(z,t)$ and $v(0)\ge \tilde{v}(0)$. Then $v(t)\ge \tilde{v}(t)$ for all $t\in [0,T]$.
\end{lemma}

\begin{proof}
Let $u(t)=\tilde{v}(t)-v(t)$, then $u(0)=0$. If there exists $\xi>0$ such that $u(\xi)>0$ then
$$
u'(\xi)=\widetilde{F}(\tilde{v}(\xi),\xi)-F(v(\xi),\xi)< \widetilde{F}(v(\xi),\xi)-F(v(\xi),\xi)\le 0.
$$
Since $u(0)\le 0$, $u(\xi)>0$ and $u'(\xi)<0$, $u(t)$ has a local maximum in $(0,\xi)$. Let $0<\eta<\xi$ be a maximum point. Then $u(\eta)>0$ and $u'(\eta)=0$, i.e. $\tilde{v}(\eta)>v(\eta)$ and
$$
\widetilde{F}(\tilde{v}(\eta),\eta)=\tilde{v}'(\eta) =v'(\eta)=F(v(\eta),\eta)> F(\tilde{v}(\eta),\eta)\ge \widetilde{F}(\tilde{v}(\eta),\eta),
$$
a contradiction follows.
\end{proof}

\begin{lemma}
\label{lem:limit}
Let $F(z,t)$ be Lipschitz function in $[0,S]\times [0,\infty)$ such that
\begin{itemize}
\item[(a)] $F(0,t)<0$, $F(S,t)>0$ for all $t>0$;
\item[(b)] there exists $c>0$ such that 
$$
F(z_1,t)-F(z_2,t)\ge c(z_2-z_1) \quad\text{ for $t\ge0$ and  $0\le z_1<z_2\le S$};
$$
\item[(c)] if $0<z(t)<S$ is the unique solution of $F(z(t),t)=0$ then $\lim\limits_{t\to\infty}z(t)=\bar z$.
\end{itemize}
Then for any solution of 
$$
u'(t)=F(u(t),t), \quad 0<u(0)<S
$$ there holds $\lim\limits_{t\to\infty}u(t)=\bar z$.
\end{lemma}

\begin{proof}
By (b) $F(z,t)$ is strictly decreasing in $z$ for each $t\ge0$. It follows from the conditions (a)--(b) and the classical Clarke result \cite{Clarke} that $z(t)$ in (c) is well-defined and local Lipschitz on $[0,\infty)$. It follows from (a) that $0<u(t)<S$ for all $t\ge0$. Now, two alternatives are possible: (i) either there exists $T>0$ such that $u(t)\ne z(t)$ for $t\ge T$, or (ii) there exists $t_k\nearrow \infty$: $u(t_k)=z(t_k)$. First let (i) hold and assume without loss of generality that $u(t)<z(t)$ for $t\ge T$. Then
\begin{equation}\label{difu}
u'(t)=F(u(t),t)-F(z(t),t)\ge c(z(t)-u(t))\ge0,
\end{equation}
hence $u(t)$ is nondecreasing, therefore there exists
\begin{equation}\label{baru}
\bar u:=\lim_{t\to \infty} u(t)\le \lim_{t\to \infty} z(t)=\bar z.
\end{equation}
Combining \eqref{baru} with the monotonicity of $u(t)$ and \eqref{difu} imples
$$
\int_{t}^\infty |z(s)-u(s)|\,ds\le \frac1c(\bar u-u(t))\to 0\,\, \text{ as }t\to\infty
$$
which  implies the equality in \eqref{baru}. Next, if (ii) holds then $\lim_{k\to \infty}u(t_k)=\bar z$. Assume by contradiction that, for example, $\bar u:=\lim\sup_{t\to\infty} u(t)>\bar z$ and let $\xi_k\nearrow\infty$ be a corresponding sequence where the $\lim\sup$ is attained. Since
$$
\lim_{k\to \infty}u(t_k)=\bar z<\bar u=\lim_{k\to \infty}u(\xi_k)
$$
one can redefine the sequence $\xi_k$ such that each $\xi_k$ becomes a local maximum of $u$. This yields $0=u'(\xi_k)=F(u(\xi_k),\xi_k)$, thus $u(\xi_k)=z(\xi_k)$. Passing to limit as $k\to\infty$ yields a contradiction.
\end{proof}

\section{Period-two-points of non-increasing maps}\label{sec:period}
Let $0\in D\subset \R^{n}_+$ and $G:D\to D$ be an arbitrary map. Recall that a pair $(a,b)$, $a,b\in D$, is called a \textit{period-two-point} \cite[p.~387]{Hale}, or $(a,b)\in \Fix_2(G)$, if
\begin{equation}\label{loop}
G(a)=b, \quad G(b)=a.
\end{equation}
Any fixed point $c\in \Fix(G)$ gives rise to a trivial period-two-point $(c,c)$.

Hereinafter, we assume that $G$ is continuous and non-increasing in $D$, i.e. $G(x)\ge G(y)$ for any $x\le y$ in $D$. Note that $G$ is then automatically bounded:
\begin{equation}\label{bounn}
0\le G(x)\le G(0),\qquad \forall x\in D.
\end{equation}

Since $0\in D$, the iterations $u^0=0$,  $u^k:=G^k(0)\in D$, $k\ge1$, are well-defined, $u^1\ge u^0=0$ (an a priori estimate) and $u^2=G(u^1)\le G(u^0)=u^1$ (by virtue of the monotonicity of $G$). Hence, it follows by induction that
\begin{equation}\label{Sequ}
u^0\leq u^2\leq\ldots u^{2k}\le\ldots u^{2k+1}\le\ldots\leq u^3\leq u^1.
\end{equation}
This implies that the limits
\begin{equation}\label{inneq}
\check{0}_G:=\lim_{k\to \infty}u^{2k}\le \hat{0}_G:=\lim_{k\to \infty}u^{2k-1}
\end{equation}
exist and $(\check{0}_G,\hat{0}_G)$ is a period-two-point of $G$:
\begin{equation}\label{hin}
G(\check{0}_G)=\hat{0}_G, \quad \check{0}_G=G(\hat{0}_G)
\end{equation}

Thus obtained period-two-point is extremal as the following property shows.

\begin{proposition}
\label{pro:loop}
For any $(a,b)\in \Fix_2(G)$  there holds
\begin{equation}\label{perr}
\check{0}_G\le a,\quad b\le \hat{0}_G.
\end{equation}
In particular,
\begin{equation}\label{local}
\check{0}_G\le c\le \hat{0}_G, \qquad \forall c\in \Fix(G),
\end{equation}
and the box
\begin{equation}\label{Hbox}
[\check{0}_G,\hat{0}_G]:=\{u:\check{0}_G\le u\le \hat{0}_G\}
\end{equation}
is invariant under the mapping $G$.
\end{proposition}

\begin{proof}
Since $a\ge u^0=0$ and $G$ is a non-increasing one has
$$
u^{2k}\le G^{2k}(a)=a, \quad
u^{2k-1}\ge G^{2k-1}(a)=b, \quad \text{ for all }\,k=1,2,\ldots
$$
This readily yields \eqref{perr}. Then \eqref{local} follows from the fact that $(c,c)$ is a period-two-point for any  $c\in \Fix(G)$. The last claim of the proposition follows immediately from the monotonicity of $G$ and \eqref{hin}.
 \end{proof}

\begin{proposition}
\label{pro:est}
Let $x,y\in D$ be such that
\begin{equation}\label{xyeq}
G(y)\le x,\quad y\le G(x).
\end{equation}
Then there exists $(a,b)\in \Fix_2(G)$ such that
\begin{equation}\label{exist}
\begin{split}
a&:=\lim_{k\to \infty}y^{2k-1}=\lim_{k\to \infty}x^{2k}\ge \check{0}_G,\\ b&:=\lim_{k\to \infty}y^{2k}=\lim_{k\to \infty}x^{2k-1}\le \hat{0}_G.
\end{split}
\end{equation}
\end{proposition}

\begin{proof}
Let $y^0=y$ and $y^k=G^k(y)$, $k\ge 1$, hence \eqref{xyeq} becomes
$$
y^1\le x^0,\quad y^0\le x^1.
$$
Applying $G$ we yields $y^2\ge x^1\ge y^0$ and $x^0\ge y^1\ge x^2$.
Proceeding by induction on $k$, we obtain by virtue of \eqref{bounn}
\begin{align*}
0&\le \ldots \le x^{4}\le y^3\le x^2\le y^1 \le x^0,\\
y^0&\le x^1\le y^2\le x^3\le y^4\le \ldots \le G(0).
\end{align*}
This implies the existence of limits in \eqref{exist}. It also follows that $G(a)=b$ and $G(b)=a$, hence $(a,b)\in \Fix_2(G)$ and $a\le x$, $ y\le b$. Combining with the extremal property \eqref{perr} yields  \eqref{exist}.
\end{proof}

\section{Stratification of equilibrium points}\label{sec:equil}

Let us denote by
$\Equ$ the set of nonnegative equilibrium points (stationary solutions) of \eqref{HX1}-\eqref{HV1}. It is natural to  consider the standard  stratification
$$
\Equ=\bigcup_{J}\Equ_J,
$$
where
$$
\Equ_J=\{(x,v)\in \Equ: \,\, x_j\ne 0 \,\,\Leftrightarrow\,\, j\in J \},
$$
and $J$ runs over all subset of $\{1,2,\ldots,M\}$.
The supply point $S$ is the equilibrium resource availabilities  in the absence of any species and obviously $(0,S)$ is the only point in $\Equ_{\emptyset }$:
$$
\Equ_{\emptyset }=\{(0,S)\}.
$$

\begin{proposition}\label{pro:triv}
For an arbitrary $(0,S)\ne (x,v)\in \Equ$ there holds
\begin{equation}\label{positivity}
x>0 \quad \text{and}\quad 0\ll v\ll S.
\end{equation}
\end{proposition}

\begin{proof}
If $x=0$ then $v=S$, thus $x>0$. If some $v_i=0$ then \eqref{MM2b} yields $\phi_j(v)=0$ for all $j$, hence by \eqref{HV1}  $v_i=S_i$, a contradiction, i.e. $v\gg 0$. Finally, note that $v\le S$. If $v_i=S_i$ for some $i$ then $\sum_{j=1}^M c_{ ij} \; x_j \; \phi_j(v)=0$. By the above, there exists $x_{k}\ne 0$, therefore $\phi_k(v)=0$ implying by \eqref{MM2b} that $v\in \partial \R^{m}_+$, thus  $\phi_j(v)=0$ for all $j$. Applying the stationary condition to \eqref{HV1} we see that $v=S$, a contradiction with $v\in \partial \R^{m}_+$. Therefore $v\ll S$.
\end{proof}

Let $(x,v)\in \Equ_J$. Then $x_j=0$ if $j\not\in J$ and
$$
x_j=\mathbf{X}_j(v):=\half1{\gamma_j} (\phi_j(v) -\mu_j)_+>0\, \quad \text{for all }\,i\in  J,
$$
where $w_+=\max (0,w)$, therefore $v$ is determined uniquely by
\begin{equation}\label{split}
\begin{split}
v_i&=S_i-\sum_{j\in J} \frac{c_{ij}}{D_i}\,   \mathbf{X}_j(v) \phi_j(v)=:(\mathbf{F}_J(v))_i.
\end{split}
\end{equation}
Extend $\textbf{F}_J$ by  $\mathbf{F}_{\emptyset}(v):= S$.
In the present setting, if $(x,v)\in \Equ_J$ then $v$ solves the fixed point problem
\begin{equation}\label{fixedJ}
v=\mathbf{F}_J(v),
\end{equation}
and
\begin{equation}\label{fixedX}
x_j=
\left\{
\begin{array}{ll}
0& \text{if $j\not\in J$}\\
\mathbf{X}_j(v)& \text{if $j\in J$}\\
\end{array}
\right.
\end{equation}
The converse is not necessarily true: if $v$ is a solution of \eqref{fixedJ} and $x$ is defined by \eqref{fixedX} then $(x,v)$ is an equilibrium point in $\Equ_{J'}$ for some $J'\subset J$. Indeed, it might happen that $\phi_j(v)\le \mu_j$, i.e. $x_j=0$ for some $j\in J$. On the other hand, if $J\ne\emptyset$ then necessarily  $J'\ne\emptyset$ because if  $x_j=0$ for all $j$ then  $(x,v)=(0,S)$, but $\mathbf{F}_J(S)\ll S$ in view of \eqref{MM2b}, a contradiction with \eqref{fixedJ}.

To distinguish this situation, we denote by
$$
\widetilde {\Equ}_{J}= \text{ the set of solutions $(x,v)$ of \eqref{fixedJ} and \eqref{fixedX}}.
$$
Then $\widetilde {\Equ}_{\emptyset}=\Equ_{\emptyset}$, and the above argument yields that for any $J\ne \emptyset$
\begin{equation}\label{emptyE}
\Equ_{J}\subset \widetilde {\Equ}_{J}\subset \bigcup_{\emptyset\ne J'\subset J}
\Equ_{J'}
\end{equation}

Thus refined stratification $J\to \widetilde {\Equ}_{J}$ still contains information about all equilibrium points but it has better properties than $J\to {\Equ}_{J}$.

\begin{proposition}\label{pro:empty}
For any $J\ne \emptyset$, the set $\widetilde {\Equ}_{J}$ is nonempty.
\end{proposition}

\begin{proof}
Consider a modified  fixed point problem
$$
v=(\mathbf{F}_J(v))_+:=\max(\mathbf{F}_J(v),0).
$$
Then $v\to (\mathbf{F}_J(v))_+$ maps continuously  the box $[0,S]$ into itself, hence by Brouwer's theorem there exists a fixed point $v\in [0,S]$. If $v_k=0$ for some $k$ then by \eqref{MM2b} we have $\phi_j(v)=0$ for all $j$, thus $v_k=(\mathbf{F}_J(v))_k=S_k$, a contradiction. Thus $v\gg0$ and $v_k=[\mathbf{F}_k(v)]_+>0$ for all $k$, therefore in fact $v_k=\mathbf{F}_k(v)$ holds for all $k$. This proves that $v$ is a  solution of the original fixed point problem \eqref{fixedJ} and $v\gg0$. If $x$ is defined by \eqref{fixedX} then it follows that $(x,v)\in \widetilde{\Equ}_J$.
\end{proof}

\section{An auxiliary finite-dimensional fixed point problem}\label{sec:fixed}
Among all equilibrium points in $\Equ$, we shall distinguish the \textit{special} ones, namely those containing in
$$
\widetilde {\Equ}_{M}:=\widetilde {\Equ}_{\{1,2,\ldots, M\}}.
$$
Equivalently, a point $(x,v)$ is said to be a special (equilibrium) point  if and only if $v$ is a solution of the fixed point problem
\begin{equation}\label{fixedp}
v=\mathbf{F}(v),\quad \mathbf{F}:=\mathbf{F}_{\{1,2,\ldots, M\}},
\end{equation}
and $x$ is given by
\begin{equation}\label{xdef}
x_j=\mathbf{X}_j(v):=\half1{\gamma_j} (\phi_j(v) -\mu_j)_{+}.
\end{equation}

By Proposition~\ref{pro:empty}, the set of special equilibrium points is nonempty.  Note also that if $(x,v)$ is an arbitrary equilibrium point of \eqref{HX1}--\eqref{HV1} with $x\gg 0$ then it is necessarily a special one because by \eqref{HX1}  $\phi_j(v)>\mu_j$ for all $j$, hence $x$ is determined by \eqref{xdef} and therefore $v$ satisfies \eqref{fixedp}.


The set of special equilibrium points $\widetilde {\Equ}_{M}=\Fix(\mathbf{F})$ reflects the complexity of large-time dynamics of the original system in the following sense. Theorem~\ref{th1} below shows that if there exists a unique global stable equilibrium point of \eqref{HX1}, \eqref{HV1}, \eqref{Idata}  then it is necessarily a special point (in this case, obviously, unique). Therefore the structure and the number of special equilibrium points plays a crucial role in the large-time dynamics of \eqref{HX1}, \eqref{HV1}, \eqref{Idata} .

Thus, it is naturally to expect that the global stability will be lost if the cardinality $|\Fix(\mathbf{F})|\ge 2$. Note that if $m=1$ then  Lemma~\ref{lem:elem} easily implies that $\Omega$ consists of exactly one point: $|\Fix(\mathbf{F})|=1$. However, if $m\ge2$, the situation is more subtle as the example below shows.

\begin{example}\label{ex1}
First let us consider \eqref{Liebm} with $M=m=2$,
$$
(c_{ij})=\left(
         \begin{array}{cc}
           1 & 0 \\
           0 & 1 \\
         \end{array}
       \right),\quad (K_{ij})=\left(
         \begin{array}{cc}
           1 & \beta \\
           \beta & 1 \\
         \end{array}
       \right),\quad
       \mu_j=0, \,\,S_i=S,\,\, \frac{r_i^2}{D_i\gamma_j}=:A>S
$$
for all $i=1,2$, where $\beta>1$ to be specified later. Then $\mathbf{F}=(f(v_1,v_2),f(v_2,v_1))$, with
$$
f(x,y)=S-A\min(\frac{x^2}{(1+x)^2}, \frac{y^2}{(\beta+y)^2}).
$$
Lemma~\ref{lem:elem} easily yields  the existence of exactly one solution of \eqref{fixedp} on the diagonal $v_1=v_2$, $0<v_1<S$. We claim that there exists yet another solution in the triangle $\Delta=\{0<\beta v_1\le v_2<S\}$. Indeed, $$\mathbf{F}|_{\Delta}=(S-\frac{Av_1^2}{(1+v_1)^2}, S-\frac{Av_1^2}{(\beta+v_1)^2}),
$$
and by Lemma~\ref{lem:elem} there exists a unique  $0<\bar v_1<S$ such that  $S-\bar v_1=\frac{A\bar v_1^2}{(1+\bar v_1)^2}$. Define $\bar v_2=S-\frac{A\bar v_1^2}{(\beta+\bar v_1)^2}$. Then $(\bar v_1,\bar v_2)$ will be a desired  fixed point if we ensure that it belongs to $\Delta$. We have
\begin{equation}\label{imm}
\frac{ \bar v_2-\beta \bar v_1}{\beta-1}=A\biggl(g(\beta,\bar v_1)-\frac{S}{A}\biggr), \quad \text{where }
g(\beta,t)=\frac{t^2\bigl((t+1+\beta)^2-\beta\bigr)}{(t+1)^2(t+\beta)^2}
\end{equation}
Notice that for any $\beta>1$, $g(\beta,t)$ is an increasing function of $t>0$, $g(\beta,0)=0$ and
$$
\lim_{t\to \infty}g(\beta,t)=1>\frac{S}{A},
$$
therefore there exists a unique  $t_\beta>0$ such that $g(\beta,t_\beta)=\frac{S}{A}$. Next notice that $\frac{\partial g}{\partial \beta}<0$, hence $t_\beta$ is a decreasing continuous function of $\beta$. Since $\lim_{\beta\to\infty}g(\beta,t)\equiv 1$ uniformly on any ray $(\epsilon,\infty)$, $\epsilon>0$, we also have $$\lim_{t\to+0}t_\beta=0. 
$$Therefore there exists $\bar{\beta}$ such that $\bar v_1>t_{\bar{\beta}}$, thus $g(\bar{\beta},\bar v_1)>g(\bar{\beta},t_{\bar{\beta}})=0$ and \eqref{imm} yields $\bar v_2-\bar{\beta}\bar v_1>0$, implying our claim. Next, since $\bar v_2>\bar\beta \bar v_1>\bar v_1$, the found solution is out the diagonal. By symmetry reasons, $(\bar v_2,\bar v_1)$ is also a solution of \eqref{fixedp}. Finally, since all the three solutions are distinct, the standard continuity argument implies that \eqref{fixedp} still has three distinct solutions for $(c_{ij})=\left(\begin{array}{cc}
           1 & \epsilon_1 \\
           \epsilon_1 & 1 \\
         \end{array}
       \right)$
       and $\mu_j=\epsilon_2$ when $\epsilon_i>0$ small enough.
\end{example}

A  careful analysis shows that for $m=2$ there always holds  $|\Fix(\mathbf{F})|\le 3$ Liebig-Monod model \eqref{Liebm}. Furthermore, for any $m\ge2 $, an argument similar to Example~\ref{ex1} yields $|\Fix(\mathbf{F})|\ge m+1$ for certain sets of parameters.

%
%
%
%

Now, let us turn to the fixed point problem \eqref{fixedp}. It is naturally to study solutions of \eqref{fixedp} by virtue of iterations $\mathbf{F}^k(0)$. But \eqref{fixedp} is  non-regular in the sense that already the second iteration $\mathbf{F}^2(0)$ can be outside of $[0,S]$. Indeed, $\mathbf{F}^2_i(0)=\mathbf{F}_i(S)$ becomes negative if $\gamma_j$ or $D_i$ are small enough (alternatively, $c_{ij}$ large enough).

To refine iterations, we suitably polarize \eqref{fixedp} to get a system with the same set of fixed points. Namely, given $w\in [0,S]$ let us define $\mathbf{V}(w) \in [0,S]$ as the unique solution $v$ of the system
\begin{equation}
S_i-v_i=\sum_{j=1}^{M} \frac{c_{ij}}{D_i}   \mathbf{X}_j(w) \phi_j(w_1,\ldots,w_{i-1},v_i,\ldots,w_m)
,\quad i=1,\ldots,m.
\label{Starvd}
\end{equation}
Note that each equation of system \eqref{Starvd} contains a single unknown variable $v_i$, thus Lemma~\ref{lem:elem}  implies that for all $i$ a unique solution $v_i$ of \eqref{Starvd} exists and $0<v_i\le  S$. Therefore $0\ll \mathbf{V}(w)\le S$. Also, by the survivability condition \eqref{survav} $\mathbf{X}_j(S)=\half1{\gamma_j} (\phi_j(S) -\mu_j)>0$, hence
\begin{equation}\label{VS}
0\ll \mathbf{V}(S)\ll S.
 \end{equation}
Furthermore, the second part of Lemma~\ref{lem:elem} implies that $\mathbf{V}(w)$ is \textit{non-increasing}:
 $$
 w_1\le w_2\,\,\Rightarrow\,\,\mathbf{V}(w_1)\ge \mathbf{V}(w_2).
 $$
Now, if $v$ solves \eqref{fixedp} then by the uniqueness of solution of \eqref{Starvd} one has
\begin{equation}\label{K2}
v=\mathbf{V}(v).
\end{equation}
Conversely, if $v$ is a solution of \eqref{K2} then it also solves \eqref{fixedp}. Thus, in the present setting, the fixed point problem \eqref{fixedp} is completely equivalent to  \eqref{K2}:
$$
\Fix(\mathbf{F})=\Fix(\mathbf{V}).
$$
The main advantage of $\mathbf{V}$ with respect to $\mathbf{F}$ is that by its definition,
$$
\mathbf{V}:[0,S]\to [0,S].
$$

Now, with $\mathbf{V}$ in hands we apply the technique of Section~\ref{sec:period}. Namely, using the definition \eqref{inneq}, we see that starting with $u^0=0$, the even and odd iterations converge respectively to
\begin{equation}\label{inneq1}
\lim_{k\to \infty}\mathbf{V}^{2k}(0)=:\check{0}_{\mathbf{V}}\le \hat{0}_{\mathbf{V}}:=\lim_{k\to \infty}\mathbf{V}^{2k-1}(0).
\end{equation}
In particular,
$$
(\check{0}_{\mathbf{V}},\, \hat{0}_{\mathbf{V}})\in \Fix_2(\mathbf{V}),
$$
and, furthermore, $(\check{0}_{\mathbf{V}},\, \hat{0}_{\mathbf{V}})$ possesses the extremal property in Proposition~\ref{pro:loop}. In particular, it follows from \eqref{local} that
\begin{equation}\label{Hbox1}
\check{0}_{\mathbf{V}}\le v\le \hat{0}_{\mathbf{V}},\quad \forall v\in \Fix(\mathbf{V}).
\end{equation}
This imemdiately yields

\begin{corollary}
If the equality 
\begin{equation}\label{identity}
\check{0}_{\mathbf{V}}=\hat{0}_{\mathbf{V}}
\end{equation}
holds then there exists a unique special equilibrium point, i.e. $$|\Fix(\mathbf{F})|=|\Fix(\mathbf{V})|=1.
$$
\end{corollary}

Conversely, \eqref{Hbox1} implies that the cardinality of fixed points $|\Fix(\mathbf{F})|$ is an obstacle for the coincidence relation \eqref{identity}. Furthermore, Example~\ref{ex1} above shows that for certain values of parameters of our system one has $|\Fix(\mathbf{F})|>1$, thus, one cannot expect in general the coincidence in  \eqref{identity}. Therefore it is important to know when \eqref{identity} holds. One such sufficient condition is presented below.

\begin{proposition}\label{cor1a}
Let $L_j$ be the $L^\infty$-Lipschiz constant of $\phi_j$. If $m=1$ and
\begin{equation}\label{rhogammamu}
\rho_1:=\sum_{j=1}^M\frac{\mu_j c_{1j}L_j}{D_1\gamma_j}<  1
\end{equation}
or $m\ge 2$ and
\begin{equation}\label{rhogamma}
\rho_m:=\max_{1\le i\le m} \sum_{j=1}^M\frac{(2\phi_j(S)-\mu_j)c_{ij}L_j}{D_i\gamma_j}\le 1
\end{equation}
then \eqref{identity} holds.
\end{proposition}

\begin{proof}
Assume by contradiction that $
\hat{0}_{\mathbf{V}}=(\eta_1,\ldots,\eta_m)>\check{0}_{\mathbf{V}} =(\xi_1,\ldots,\xi_m)$ and rewrite \eqref{hin} as
\begin{equation}\label{bbal}
\left\{
\begin{array}{ll}
S_i-\xi_i=
\sum_{j=1}^M\frac{c_{ij}}{D_i\gamma_j}(\phi_j(\eta)-\mu_j)_+\phi_j(\eta^{(i)}),
\qquad \eta^{(i)}=(\eta_1,\ldots,\xi_i,\ldots\eta_m),\\
S_i-\eta_i=
\sum_{j=1}^M\frac{c_{ij}}{D_i\gamma_j}(\phi_j(\xi)-\mu_j)_+\phi_j(\xi^{(i)}), \qquad \,\xi^{(i)}=(\xi_1,\ldots,\eta_i,\ldots\xi_m)
\end{array}
\right.
\end{equation}
First let us consider the case $m=1$. Then \eqref{bbal} takes a simpler form
\begin{equation}\label{bbal1}
\left\{
\begin{array}{ll}
S_1-\xi_1=
\sum_{j=1}^M\frac{c_{1j}}{D_1\gamma_j} (\phi_j(\eta_1)-\mu_j)_+\phi_j(\xi_1),
\\
S_1-\eta_1=
\sum_{j=1}^M\frac{c_{1j}}{D_1\gamma_j} (\phi_j(\xi_1)-\mu_j)_+\phi_j(\eta_1).
\end{array}
\right.
\end{equation}
A simple analysis shows that for all $0\le a\le b$, $\mu\ge 0$, the following inequality is  true
$$
a(b-\mu)_+-b(a-\mu)_+=\left\{
\begin{array}{rr}
0& \text{if }a\le b\le \mu\\
a(b-\mu)& \text{if }a\le \mu\le b\\
\mu(b-a)& \text{if }\mu\le a\le b
\end{array}
\right.
\le \mu(b-a)
$$
therefore, taking into account that $\phi_j(\xi_1)\le \phi_j(\eta_1)$ and subtracting relations in \eqref{bbal1} one obtains
$$
0<\eta_1-\xi_1\le \sum_{j=1}^M\frac{c_{1j}\mu_j}{D_1\gamma_j}(\phi_j(\eta_1)-\phi_j(\xi_1)) \le \rho_1 (\eta_1-\xi_1)< (\eta_1-\xi_1),
$$
 a contradiction follows.

Now, let $m\ge 2$. Since $\xi^{(i)}\ge \xi$ and $\eta^{(i)}\le \eta$, we obtain on subtracting equations in \eqref{bbal} that
\begin{equation}\label{lipp}
\eta_i-\xi_i\le \sum_{j=1}^M\frac{c_{ij}}{D_i\gamma_j}(f_j(\phi_j(\eta))-f_j(\phi_j(\xi))),
\end{equation}
where $f_j(x)=x(x-\mu_j)_+$ has the Lipschiz constant $(2b-\mu_j)$ on $0\le x\le b$. Combining this with the fact that $\phi_j(\eta)\le\phi_j(S)$ and $\mu_j< \phi_j(S)$ we obtain from \eqref{lipp} and using the definition of $\rho$ that
$$
\|\eta-\xi\|_{\infty}< \rho \cdot \|\eta-\xi\|_{\infty}\le \|\eta-\xi\|_{\infty},
$$
where $\|x\|_{\infty}=\max_{1\le i\le m}|x_i|$. This immediately yields the desired contradiction.
\end{proof}


In general one has from \eqref{VS} that $\mathbf{V}(S)\ll S$ and $\mathbf{V}^2(S)\le S$, hence the following simple bilateral estimates hold:
\begin{equation}\label{local1}
0\ll \mathbf{V}(S)\le \check{0}_{\mathbf{V}} \le \hat{0}_{\mathbf{V}}\le \mathbf{V}^2(S)\le S.
\end{equation}
The latter estimate \eqref{local1} is optimal in general. Indeed, if  $\frac{c_{ij}}{D_i\gamma_j}$ are large enough, $\mathbf{V}(S)$ can be made arbitrarily small, for instance such that  $\phi_j(\mathbf{V}(S))\le \mu_j$, which yields $\mathbf{V}^2(S)=S$ and therefore $(\check{0}_{\mathbf{V}},\hat{0}_{\mathbf{V}})=(\mathbf{V}(S), \,S)$.

\begin{proposition}\label{pro:w}
For any $w\in [0,S]$,
\begin{equation}\label{Phi}
[\mathbf{F}(w\wedge \mathbf{V}(w))]_+\le \mathbf{V}(w)\le
\mathbf{F}(w\vee \mathbf{V}(w)).
\end{equation}
In particular,
\begin{equation}\label{Phi1}
[\mathbf{F}(S)]_+\le \mathbf{V}(S)\le
 \mathbf{V}^2(S)\le
\mathbf{F}([\mathbf{F}(S)]_+).
\end{equation}
Here $x\vee y$ (resp. $x\wedge y$) denote the vector whose $i$th coordinate is $\max(x_i,y_i)$ (resp. $\min(x_i,y_i)$).

\end{proposition}

\begin{proof}
Let $v=\mathbf{V}(w)$. Since for any $i$, $w\vee v\le (w_1,\ldots,w_{i-1},v_i,\ldots,w_m)$, it follows from the monotonicity of $\mathbf{F}$ and \eqref{Starvd} that
\begin{align*}
\mathbf{F}_i(w\vee v)&\ge \mathbf{F}_i(w_1,\ldots,w_{i-1},v_i,\ldots,w_m)\\
&=S_i-\sum_{j=1}^{M} \frac{c_{ij}}{D_i}   \mathbf{X}_j(w) \phi_j(w_1,\ldots,w_{i-1},v_i,\ldots,w_m)\\
&=v_i=\mathbf{V}_i(w),
\end{align*}
which yields the right inequality in \eqref{Phi}. The left one follows by a similar argument from $w\wedge v\ge (w_1,\ldots,w_{i-1},v_i,\ldots,w_m)$ and the fact that $\mathbf{V}(w)\ge0$. Then \eqref{Phi1} follows from $0\ll \mathbf{V}(S)\le \mathbf{V}^2(S)\le S$  and \eqref{Phi}.
\end{proof}

\section{Bilateral estimates} \label{sec:main}
As it was pointed out before, Example~\ref{ex1} shows that a priori the asymptotic behaviour of solutions to \eqref{HX1}--\eqref{HV1} can be rather complicated for $m\ge 2$. On the other hand, the result below shows that the global dynamics is completely controlled by the finite-dimensional fixed point problem \eqref{fixedp} and the characteristic parameters in \eqref{hin}.

\begin{theorem}\label{th1}
Let $(x(t), v(t))$ be the solution of \eqref{HX1}, \eqref{HV1}, \eqref{Idata}. Then in notation of Section~\ref{sec:period} and \ref{sec:equil}:
\begin{align}
\check{0}_{\mathbf{V}}\le \lim\inf_{t\to\infty}v(t)&\le \lim\sup_{t\to\infty}v(t)\le \hat{0}_{\mathbf{V}},\label{bilateral1}\\
\mathbf{X}(\check{0}_{\mathbf{V}})\le \lim\inf_{t\to\infty}x(t)&\le \lim\sup_{t\to\infty}x(t)\le \mathbf{X}(\hat{0}_{\mathbf{V}}).\label{bilateral2}
\end{align}
In particular, if $\check{0}_{\mathbf{V}}=\hat{0}_{\mathbf{V}}$ then all solutions of \eqref{HX1}--\eqref{HV1} converge to a unique special equilibrium point.
\end{theorem}

\begin{proof}
%

As the first step we reformulate the original system as an appropriate  integral equation for unknown function $v(t)$. Let $w(t):[0,\infty)\to [0,S]$ be a a continuous vector-function with $w(0)=v(0)$ and having a limit $\lim_{t\to \infty}w(t)=\bar w$. Then
\begin{equation}
  \mathscr{X}_i(w)(t)=x_i(0)\left(e^{- \int_0^t (\phi_j(w(s))-\mu_j) ds} + \gamma_j x_i(0) \int_0^t e^{- \int_{t_1}^t (\phi_j(w(s))-\mu_j) ds}  dt_1\right)^{-1},
\label{AppX}
\end{equation}
solves \eqref{HX1} with $v(t)$ replaced by $w(t)$. Clearly, $\mathscr{X}(w)(t)$ is a nondecreasing functional of $w$, $\mathscr{X}(w)(0)=x(0)$ and one can readily verify that
\begin{equation}\label{xlimit}
\lim_{t\to\infty}\mathscr{X}_i(w)(t)=\half1{\gamma_j} (\phi_j(v) -\mu_j)_{+}= \mathbf{X}_i(\bar w).
\end{equation}
Next, let  $\mathscr{V}(w)(t)$ denote the solution $u(t)$ of the system below (obtained from \eqref{HV1} with $x(t)$ replaced by \eqref{AppX}):
\begin{equation}\label{HV1a}
\begin{split}
     \frac{du_i}{dt}&=D_i(S_i -u_i)   -  \sum_{i=1}^M c_{ ij} \; \phi_j(w_1,\ldots,w_{i-1},u_i,\ldots,w_m)\mathscr{X}_j(w)(t)\\ u_i(0)&=v_i(0), \qquad i=1,\dots, m.
\end{split}
\end{equation}
Let $C^1_S[0,T]$ denote the set of $C^1$-vector-functions $u(t)$ on $[0,T]$ such that $0\le v(t)\le S$. Then
\begin{equation}\label{Qmap}
\mathscr{V}:C^1_S[0,T]\to C^1_S[0,T], \quad \forall T>0.
\end{equation}
Next, note that $\mathscr{V}(w)$ is a non-increasing functional of $w$. Indeed, let $0\le w(t)\le \widetilde{w}(t)$ for all $t\ge0$, and let $u_i(t)$ and $\widetilde{u}_i(t)$ be the corresponding solutions of \eqref{HV1a}. Denote by $F_i(u_i(t),t)$ and $\widetilde{F}_i(\widetilde{u}_i(t),t)$ the right hand side of \eqref{HV1a} corresponding to $w(t)$ and $\widetilde{w}(t)$ respectively.  Then the $F_i(z,t)$ and $\widetilde{F}_i(z,t)$ satisfies the conditions of Lemma~\ref{lem:elem2} and $u_i(0)=\widetilde{u}_i(0)$, therefore $u_i(t)\ge \widetilde{u}_i(t)$ for all $t$, as desired.

Furthermore, we claim that
\begin{equation}\label{vlimit}
\lim_{t\to\infty}\mathscr{V}(w)(t)= \mathbf{V}(\bar w).
\end{equation}
Indeed, rewrite \eqref{HV1a} as $u_i'(t)=F_i(u_i(t),t)$, where
$$
F_i(z,t):=D_i(S_i -z)   -  \sum_{j=1}^M c_{ ij} \; \phi_j(w_1(t),\ldots,z,\ldots,w_m(t))\mathscr{X}_j(w)(t).
$$
Then $F_i(z,t)$ obviously satisfies conditions (a) and (b) of Lemma~\ref{lem:limit} with $c=D_i$. To verify (c), note that by \eqref{xlimit} for any $z\in [0,S]$:
$$
\bar F_i(z):=\lim_{t\to\infty} F_i(z,t)=D_i(S_i -z)   -  \sum_{j=1}^M c_{ ij} \; \phi_j(\bar w_1,\ldots,z,\ldots,\bar w_m)\mathbf{X}_j(\bar w).
$$
Comparing the latter expression with \eqref{Starvd}, we conclude that $z=\mathbf{V}_i(\bar w)$ is the unique root of $\bar F_i(z)=0$ in $[0,S]$. Now, let $0\le z_i(t)<S$ be the unique solution  of $F_i(z_i(t),t)=0$, $t\ge 0$. Suppose that $t_k\nearrow\infty$ realizes $\bar z:=\lim\sup_{t\to\infty} z_i(t)$. Then
$$
0=\lim_{k\to \infty}F_i(z_i(t_k),t_k)=\bar F_i(\bar z) \quad \Rightarrow \quad \bar z=\mathbf{V}_i(\bar w).
$$
Similarly one shows that $\mathbf{V}_i(\bar w)=\lim\inf_{t\to\infty} z_i(t)$. Thus, $\lim_{t\to\infty} z_i(t)=\mathbf{V}_i(\bar w)$ exists, as desired. Applying Lemma~\ref{lem:limit} yields \eqref{vlimit}.

In the present setting, if $(x(t),v(t))$ is the solution of \eqref{HX1}, \eqref{HV1}, \eqref{Idata} then $v=v(t)$  satisfies  the  fixed point problem
\begin{equation}
\label{Appw}
v=\mathscr{V}(v),
\end{equation}
then $x=x(t)$ is recovered by
\begin{equation}\label{xw}
x=\mathscr{X}(v).
\end{equation}

Now we show that $v(t)$ can be  obtained as the limit of  iterations
$$
v^{k}(t)=\mathscr{V}^k(v^0)(t), \quad k\ge1, \quad \text{where }v^0(t)\equiv 0.
$$
As $\mathscr{V}$ is non-increasing and $\mathscr{V}(v)=v$, one has
$$
0=v^0\le v\le v^{1}\le S.
$$
Since $\mathscr{V}^2$ is non-decreasing and by \eqref{Qmap} $v^2\ge 0=v^0$,  one readily obtains
$$
v^0\leq v^2\leq\ldots v^{2k}\leq \ldots v\le\ldots \leq v^{2k-1}\leq\ldots\leq v^3\leq v^1
$$
and $v^k\in C^1_S[0,\infty)$. For any fixed $T>0$, the operator $\mathscr{V}:C^1_S[0,T]\rightarrow C^1_S[0,T]$ is compact, hence both the odd $v^{2k-1}(t)$ and even $v^{2k}(t)$ terms  conerge in $C^1[0,T]$, therefore the following limits are well-defined for any $t\ge 0$:
\begin{equation}\label{vita1}
\check{v}(t)=\lim_{k\to\infty}v^{2k}(t),\;\;\;
\hat{v}(t)=\lim_{k\to\infty}v^{2k+1}(t),
\end{equation}
and
\begin{equation}\label{period}
\mathscr{V}(\check{v})=\hat{v}\quad \mathscr{V}(\hat{v})=\check{v}.
\end{equation}
 Since $v^0=0\le v$ we also have by \eqref{Appw} that $v^{2k}\le v\le v^{2k+1}$, thus implying
\begin{equation}\label{vnepod}
\check v(t)\le v(t)\le \hat v(t), \qquad
\check x(t)\le x(t)\le \hat x(t),
\end{equation}
where $\hat{x}=\mathscr{X}(\hat{v})$, $\check{x}=\mathscr{X}(\check{v})$, and $(\check{x},\hat{v})$ and $(\hat{x},\check{v})$ solve respectively
\begin{align*}
(\check{x},\hat{v}):&\quad \frac{d\check{x}_j}{dt}=\check{x}_j ( \phi_j(\check{v})- \mu_j  -  \gamma_{j} \check{x}_j), \qquad
\frac{d\hat{v}_i}{dt}=D_i(S_i -\hat{v}_i)   -  \sum_{j=1}^M c_{ij} \; \check{x}_j \; \phi_j(\hat{v}),\\
(\hat{x},\check{v}):&\quad \frac{d\hat{x}_j}{dt}=\hat{x}_i ( \phi_j(\hat{v})- \mu_j -  \gamma_{j} \hat{x}_j), \qquad
\frac{d\check{v}_i}{dt}=D_i(S_i -\check{v}_i)   -  \sum_{j=1}^M c_{ij} \; \hat{x}_j \; \phi_j(\check{v}),
\end{align*}
Taking the difference  yields that $(\xi,\eta):=(\hat{x}-\check{x},\hat{v}-\check{v})$ satisfies a homogeneous linear system of ODEs with  bounded coefficients (recall that $\phi_j$ are  Lipschiz). Since $(\xi,\eta)$ has the zero Cauchy data we conclude by uniqueness for the Cauchy problem and \eqref{vnepod} that $\check{v}(t)= \hat{v}(t)=v(t)$ and $\check{x}(t)= \hat{x}(t)=x(t)$. In summary, for any fixed $t>0$ one has
\begin{equation}\label{followsby}
x(t)=\lim_{k\to \infty} x^k(t),\qquad v(t)=\lim_{k\to \infty} v^k(t),
\end{equation}
 where by \eqref{vlimit} $\bar v^k:=\lim_{t\to \infty} v^k(t)=\mathbf{V}^k(0)$, $\bar v^0=0.$
Applying the results of Section~\ref{sec:period} to \eqref{vita1} yields
$\lim_{k\to \infty} \bar v^{2k}=\check{0}_{\mathbf{V}},$ $\lim_{k\to \infty} \bar v^{2k-1}=\hat{0}_{\mathbf{V}},$
which proves \eqref{bilateral1}. Similarly, \eqref{bilateral2} follows from \eqref{xlimit} and \eqref{followsby}.
\end{proof}

Combining the obtained estimates with Corollary~\ref{cor1a} implies the following global stability result.

\begin{corollary}[\textbf{Global stability}]\label{cor1}
If $m=1$ and \eqref{rhogammamu} holds or $m\ge 2$ and \eqref{rhogamma} holds then  \eqref{HX1}-\eqref{HV1} is globally stable: any solution with Cauchy data \eqref{Idata} converges to a unique equilibrium point $\check{0}_{\mathbf{V}}=\hat{0}_{\mathbf{V}}$.
\end{corollary}

%
%

Numerical simulations in  \cite{Huis1999} show that certain solutions of the standard model with $\gamma_j=0$ and $m\ge 3$ have periodic (chaotic) dynamics. Corollary~\ref{cor1} above shows that if the self-limitation constants $\gamma_j$ or dilution rates $D_i$ are large enough, the global behaviour of the modified model becomes stable for any choice of $m$ and $M$.

In fact, one can chose the parameters of the system such that the strong persistence holds, see the corollary below. To present our result we need to define an analogue of the lowest break-even concentration $R^*$ in \eqref{rstarr} for general response functions $\phi_j$. Let us consider the set
\begin{equation}\label{set}
\mathscr{R}:=\{v\in [0,S]: \phi_j(v)> \mu_j \,\,\, \text{for all }j\}.
\end{equation}
Note that by \eqref{survav}, $\mathscr{R}\ne \emptyset$.

\begin{corollary}[Strong persistence]\label{cor2}
In notation of Corollary~\ref{cor1}, there exists $\rho_0=\rho_0(\mathscr{R})>0$ such that if $\rho\le \rho_0$
then any solution of \eqref{HX1}-\eqref{HV1}-\eqref{Idata} converges to a unique equilibrium point with
$$
\lim_{t_\to \infty}x_i(t)>0\quad \text{ for all }1\le i\le M.
$$
\end{corollary}
\begin{proof}
By \eqref{survav}, $S\in \mathscr{R}$, therefore the number
$$
\delta:=\sup\{t\ge 0: (S_1-t,\ldots,S_m-t)\in \mathscr{R}\}
$$
is well defined and positive. Since $\delta\le \|S\|_\infty$, we have $\rho_0:=\delta/\|S\|_{\infty}\le 1$. If $\rho\le \rho_0$ then by Corollary~\ref{cor1} any solution with Cauchy data \eqref{Idata} converges to a unique equilibrium point $0\ll \xi\ll S$ satisfying \eqref{fixedp}. We have for all $1\le i\le m$
\begin{equation}\label{delll}
\begin{split}
S_i-\xi_i&=\sum_{j=1}^M \frac{c_{ij}}{D_i\gamma_j}(\phi_j(\xi)-\mu_j)_+\phi_j(\xi)\equiv
\sum_{j=1}^M \frac{c_{ij}}{D_i\gamma_j}[f_j(\phi_j(\xi))-f_j(0)]\\
&< \rho\|\xi\|_\infty\le \frac{\delta \|\xi\|_\infty}{\|S\|_{\infty}}<\delta
\end{split}
\end{equation}
Therefore  $\xi\in \mathscr{R}$, implying by \eqref{set} and \eqref{xdef} that $\lim_{t\to\infty}x_j(t)=\mathbf{X}_j(\xi)>0$ for all $j$, as desired.
\end{proof}

In general, one has from \eqref{bilateral1}, \eqref{local1} and \eqref{Phi1} the following explicit a priori estimate.

\begin{corollary}
Let $(x(t), v(t))$ be the solution of \eqref{HX1}, \eqref{HV1}, \eqref{Idata} and let the survivability condition \eqref{survav} holds.
Then
\begin{align*}
[\mathbf{F}(S)]_+\le \lim\inf_{t\to\infty}v(t)&\le \lim\sup_{t\to\infty}v(t)\le \mathbf{F}([\mathbf{F}(S)]_+)\\
\mathbf{X}([\mathbf{F}(S)]_+)\le \lim\inf_{t\to\infty}x(t)&\le \lim\sup_{t\to\infty}x(t)\le \mathbf{X}(\mathbf{F}([\mathbf{F}(S)]_+)).
\end{align*}
where $\mathbf{F}_i(S)=S_i-\sum_{j=1}^M\frac{c_{ij}}{D_i\gamma_j} (\phi_j(S)-\mu_j)\phi_j(S)$,  $1\le i\le m$ and $\mathbf{X}$ is defined by \eqref{xdef}.

\end{corollary}

\section{Discussion}\label{sec:dis}

In this paper we established sufficient conditions for the global stability and persistence of a chemostat like model with self-limitations. For the Liebig-Mondoc model \eqref{Liebm} one has $L_j= r_j/\min_{i}\{K_{ij}\}$ and $\phi_j(S)\le r_j$. It is interesting to compare our result with simulations in \cite{Huis1999} rigourously approved in \cite{Li2001}, see especially Section~5 there. In that example, Huisman and Weissing assume in the present notation that $m=M=3$, $S_j=10$, $r_j=1$, $D_j=0.25$ for all three species and matrices $K_{ij}$ and $c_{ij}$ be chosen as in \cite[p.~38]{Li2001}.  Then if $\gamma_j=0$ then  Theorem~3.1 in \cite{Li2001} implies the existence of a nontrivial periodic oscillation. On the other hand, it follows from \eqref{rhogamma} that  if  $\gamma_j\ge 1.64$, $i=1,2,3$, then any solution is global stable, for arbitrary positive initial data. In general, given arbitrary data, \eqref{rhogamma} explicitly defines the critical values $\gamma_j^*$ such that the system is globally stable for $\gamma_j>\gamma_j^*$.





\def\cprime{$'$}

\end{document}